\documentclass[11pt]{amsart}
\usepackage{geometry}                % See geometry.pdf to learn the layout options. There are lots.
\usepackage{amssymb}
\usepackage{epsfig}
\usepackage{graphics,color,epic}

\input{epsf.sty}

\newcommand\vpil[1]{\overset{\leftarrow}{#1}}
\newcommand\hpil[1]{\overset{\rightarrow}{#1}}

\def\nto{\mathrel{\nrightarrow}}

\newtheorem{theorem}{Theorem}[section]
\newtheorem{lemma}[theorem]{Lemma}

\newtheorem{corollary}[theorem]{Corollary}

\newtheorem{remark}[theorem]{Remark}

\begin{document}

\title{A note on correlations in \\ randomly oriented graphs}

\author{Svante Linusson}  \thanks{Svante Linusson is a Royal Swedish Academy of Sciences Research Fellow supported by a grant from the Knut and Alice Wallenberg Foundation. I thank Institut Mittag-Leffler for its hospitality.}
\address{Department of Mathematics, KTH-Royal Institute of Technology, 
  SE-100 44, Stockholm, Sweden.}

\email{linusson@math.kth.se}

\date{May 20, 2009}

\begin{abstract}
Given a graph $G$, we consider the model where $G$ is given a random orientation by giving each edge a random direction. It is proven that for $a,b,s\in V(G)$, the events $\{s\to a\}$  and $\{s\to b\}$  are positively correlated. This correlation persists, perhaps unexpectedly, also if we first condition on $\{s\nto t\}$  for any vertex $t\neq s$. 
With this conditioning it is also true that $\{s\to b\}$  and $\{a\to t\}$  are negatively correlated.

A concept of increasing events in random orientations is defined and a general inequality corresponding to Harris inequality is given.

The results are obtained by combining a very useful lemma by Colin McDiarmid which relates random orientations with edge percolation, with results by van den Berg, H\"aggstr\"om, Kahn on correlation inequalities for edge percolation. 

The results are true also for another model of randomly directed graphs.
\end{abstract}

\maketitle

\section{Introduction} \label{S:Intro}
For a given randomly directed graph and two vertices $a,s$ let $\{s\to a\}$ denote the event that there exists a directed path from $s$ to $a$. A natural question to ask is whether the existence of directed paths between various pairs of vertices are positively or negatively correlated. For example, if $a,b,s$ are vertices of the graph, Corollary \ref{C:Twopaths} states that $\{s\to a\}$ and $\{s\to b\}$ are positively correlated, which means that

\begin{equation} \label{e:Twopaths}
 P_O(s\to a)P_O(s\to b) \le P_O(s\to a,s\to b).
\end{equation}

The model $O$ means that we have oriented each edge of the graph $G$ independently with probability $1/2$ for each direction. In this note we discuss correlation inequalities in this model. Questions of this type have been studied for a long time in the context of edge percolation in undirected graphs. The purpose of this note is to draw attention to the possibility to use the model $O$ to prove results for edge percolation and vice versa. It follows from results by Colin McDiarmid \cite{CM} on so called clutter percolation. It gives for instance equality for the probability distributions of the open cluster around a vertex in edge percolation and out-cluster, i.e. the set of reachable nodes, in model $O$, see Lemma \ref{L:equal1}. An extension to the clusters around two vertices is given in Lemma \ref{L:equal2}. Theorem \ref{T:OrientedHarris} is an analog to Harris inequality for random orientations of graphs. 

The positive correlation in \eqref{e:Twopaths} is intuitively easy to grasp. Much less intuitively clear is the fact that this correlation persists also if we first condition on  $\{s\nto t\}$ for any vertex $t\neq s$, that is

\begin{equation} \label{e:vdBK}
P_O(s\to a | s\nto t)P_O(s\to b | s\nto t) \le P_O(s\to a,s\to b | s\nto t),
\end{equation}
given as Corollary \ref{C:vdBK}. The following negative correlation inequality should be more intuitively clear

\begin{equation} \label{e:vdBHK}
P_O(a\to t | s\nto t)P_O(s\to b | s\nto t) \ge P_O(a\to t,s\to b | s\nto t),
\end{equation}
which is given as Theorem \ref{T:vdBHK2}. These inequalities are transfered from edge percolation results 
in \cite{vdBK} and \cite{vdBHK} using the lemmas of Section \ref{S:Lemmas}.

\medskip
The proper definitions of the three models involved are as follows. $G=(V,E)$ is a fixed (undirected) graph.

\noindent\textbf{Model $E^p$} (Edge percolation): Let $0\le p\le 1$. Every edge in $G$ exists with probability $p$ independently of the other edges.

\smallskip
For a vertex $v\in V(G)$, let $C_v(G)\subset V(G)$ be the (random) cluster around $v$, i.e. all vertices $u$ for which there is a path of existing edges between $v$ and $u$. The $G$ will be dropped if it is clear which graph is considered. Note that this notation differs from what seems to be customary, since here $C_v$ is a set of vertices, and not a set of edges.

\medskip
\noindent\textbf{Model $O$} (Random orientation): Every edge in $G$ is directed either way with probability $1/2$ independent of the other edges. 

\smallskip
Model $O$ is one natural way to obtain a randomly directed graph. It has been considered previously in for instance \cite{CM,GG01,SL1}. Another possibility is the following.

\medskip
\noindent\textbf{Model $D^p$} (Directed edge percolation): Every edge in $G$ is replaced by two edges with opposite directions. Then each of the directed edges will exist with probability $p$ independent of all other edges.  

In both models $O$ and $D^p$ we let for a vertex $v\in V(G)$ the {\bf out-cluster} $\hpil{C}_v(G)\subset V(G)$ be the (random) set of all vertices $u$ for which there is a directed path from $v$ to $u$. Let also the {\bf in-cluster} $\vpil{C}_v(G)\subset V(G)$ be the (random) set of all vertices $u$ for which there is a directed path from $u$ to $v$. Note that by convention $v\in \vpil{C}_v(G)\cap \hpil{C}_v(G)$.

\begin{remark} 
I was originally inspired by the so called bunkbed conjecture due to Kasteleyn, Remark 5 in \cite{vdBK}. For any finite graph $G$ let the bunkbed graph be $\tilde{G}=G\times K_2$, where $K_2$ is the graph on two vertices $0,1$ with one edge. The conjecture states that for any vertices $u,v\in V(G)$,  
$P_{E^p}\left( (v,0)\in C_{(u,0)(\tilde{G})}\right) \ge P_{E^p}\left( (v,1)\in C_{(u,0)(\tilde{G})}\right)$ for any $0\le p\le 1$. Lemma \ref{L:equal1} gives an equivalent inequality for random orientations of $\tilde G$, but I was not able to use that to prove the conjecture.
See also \cite{SL1}, where it is proved for outer planar graphs and 
H\"aggstr\"om \cite{OH2}, who coined the term bunkbed conjecture.
\end{remark} 

\medskip

A natural question is to study also the events $\{a \to s\}$ and $\{s \to b\}$. My original intuition suggested to me that in model $O$ they are always negatively correlated. However, Sven Erick Alm \cite{AL1} found that a simple counterexample is the graph on four vertices with all edges except $ab$. This question is  studied further in \cite{AL1}, \cite{AL2}. Note that in the model $D^p$ they are always positively correlated by Harris inequality, see Theorem \ref{T:Harris} below.

\medskip
{\bf Acknowledgement:} I thank Sven Erick Alm and Olle H\"aggstr\"om for very valuable discussions and Jeff Kahn for pointing out the reference \cite{CM}.

\section{The main lemmas} \label{S:Lemmas}

For this section we fix the notation $q:=1-p$. The following lemma is a special case of the theory on clutter percolation by McDiarmid \cite{CM}.
 
\begin{lemma} \label{L:equal1}
For any (locally finite) graph $G$ and any vertex $u\in V(G)$ and finite set $U\subset V(G)$, $u\in U$, we have
\[ P_{E^{1/2}}(C_u=U)=P_O(\hpil{C}_u=U)=P_{D^{1/2}} (\hpil{C}_u=U) \text{ and }
\]
\[ P_{E^{p}}(C_u=U)=P_{D^{p}}(\hpil{C}_u=U), \quad \text {for any $0\le p\le 1$}.
\]
\end{lemma}

Even though this lemma is in my opinion very striking and surprising at first sight, it seems not to have received enough recognition and not been used as much as one could expect. The identity of the cluster distribution for models $O$ and $E^{1/2}$ were used in for instance \cite{SL1} and \cite{GG} on page 369. We give a short proof which we will generalize below. 

\begin{proof}
The proof uses induction over $|U|=k$ and we prove that the probabilities can be computed using the same recursion. If $k=1$ we have $P_{O}(\hpil{C}_u=U)=\frac{1}{2}^{\deg(u)}$
and $P_{E^{p}}(C_u=U)=P_{D^{p}}(\hpil{C}_u=U)=q^{\deg(u)}$.

For the inductive step, assume $|U|\ge 2$ and take a vertex $v\in U, v\neq u$. 
First consider the model $E^p$. If, for a given graph with $C_u(G)=U$, we remove $v$ and all its edges,  we will get $C_u(G\setminus {v})=U_1$ for some set $U_1\subseteq U\setminus \{v\}$. Let $r$ be the number of edges between $U_1$ and $v$ in $G$. Note that at least one of these edges must exist and that there must be paths from $v$ to $U\setminus U_1$ in $G\setminus U_1$. 
We get the recursion
\[P_{E^p}\big(C_u(G)=U\big)=\sum_{U_1:u\in U_1\subseteq U\setminus \{v\}}P_{E^p}\big(C_u(G\setminus \{v\})=U_1\big)
\cdot (1-q^r)\cdot P_{E^p}\big(C_v(G\setminus U_1)=U\setminus U_1\big).
\]
If we do the same reasoning for the model $D^p$ (and $O$) we get exactly the same recursion, but the argument that at least one edge must exist is replaced with at least one of the edges directed from $U_1$ to $v$ must exist, (at least one edge must be directed from $U_1$ to $v$ in model $O$, where also $q=1/2$). The lemma follows.
\end{proof}

The following extension of Lemma \ref{L:equal1} is easy to prove with the same method, but has to the best of my knowledge not been considered before.
\begin{lemma} \label{L:equal2}
For any (locally finite) graph $G$ and any vertices $u,w\in V(G)$ and finite sets $U,W\subset V(G)$, $u\in U, w\in W, U\cap W=\emptyset$, we have
\[ P_{E^{1/2}}(C_u=U, C_w=W)=P_O(\hpil{C}_u=U, \vpil{C}_w=W)=P_{D^{1/2}}(\hpil{C}_u=U, \vpil{C}_w=W) \text{ and }
\]
\[ P_{E^{p}}(C_u=U, C_w=W)=P_{D^{p}}(\hpil{C}_u=U, \vpil{C}_w=W), \quad \text {for any $0\le p\le 1$}.
\]
\end{lemma}
\begin{proof}
The proof is by induction over $|U|+|W|$, where for $|U|+|W|=2$ we get the probability $q^{\deg(u)+\deg(w)}$ (the exponent should be lowered by 1 if there is an edge between $u$ and $w$ in $G$) for all models.
The inductive step is carried out as in the proof of Lemma \ref{L:equal1}. With the same notation and assuming that there is a vertex  $v\in U\setminus \{u\}$ we get 
\begin{align*}P_{E^p}\Big(C_u(G)=U, C_w(G)=W\Big)= &
\sum_{U_1: u\in U_1\subseteq U\setminus \{v\}}P_{E^p}\Big(C_u(G\setminus \{v\})=U_1,C_w(G\setminus \{v\})=W\Big)\\
&\cdot (1-q^r)\cdot P_{E^p}\Big(C_v(G\setminus (U_1\cup W))=U\setminus U_1\Big).
\end{align*}
Again it is not difficult to see that we get the same recursion in all three models ($q=1/2$ in $O$). If $U=\{u\}$ then we use the similar recursion with $v\in W\setminus \{w\}$. Thus the statement follows by induction and by Lemma \ref{L:equal1}.
\end{proof}

\begin{remark}
If we tried to prove a similar lemma with $\hpil{C}_u=U,\hpil{C}_w=W$, we would get trouble at the boundary if there were an edge between $U$ and $W$. In that situation we could not direct the edge in the model $O$ in any way and we would need to consider both directed edges in the model $D^p$, so we would not get equality. However, if there were no edges between $U$ and $W$ in $G$ the corresponding equalities would be true.
\end{remark}

\begin{remark}
One could also consider the mixed model where the edges of  $G$ are split into two disjoint sets $E_1$ and $E_2$. Then every edge in $E_1$ is said to exist with probability $p_1$ independently of all other edges. Every edge in $E_2$ is given one of two possible directions with equal probability. The random semi-directed graph would be a natural generalization of bond percolation and random orientation for which Lemma \ref{L:equal1} and Lemma \ref{L:equal2} would be true for $p_1=1/2$.

We may also change the model so that every edge in $G$ belongs to set $E_1$ with probability $p'$ (independently of other edges) and otherwise to $E_2$. Then $p'=1$ gives model $E^{p_1}$, letting $p'=0$ gives the model $O$ and letting $p'=1-2p(1-p)$, $p_1=p^2/p'$ we get model $D^p$. Also in this model we would get the same probability distribution for $\hpil{C}_u$ for every $p'$ if we insist on $p_1=1/2$.
\end{remark}

%***CHECK POSSIBLE CONNECTION TO ORIENTED MATROIDS***

\section{Path correlations} \label{S:Paths}
In this section we will apply Lemma \ref{L:equal1} and Lemma \ref{L:equal2} to obtain correlation inequalities in the model $O$. The results are true also for the model $D^p$, but are already known to be true. Edge percolation on directed graphs is more closely connected to ordinary percolation and the most general result below, Theorem \ref{T:OrientedvdBHK}, is for the model $D^p$ a special case of Theorem 3.1 in \cite{vdBHK}. 

For edge percolation, $E^p$, the concept of increasing events is important. Here the total space of realizations is $\Omega=\{0,1\}^{E}$, where $E$ is the edge set of the graph. For elements $\omega,\omega'\in \Omega$, write $\omega\ge \omega'$ if $\omega(e)\ge \omega'(e)$ for all $e\in E$. An event $A$ is called increasing if $\omega\ge \omega'\in A$ implies $\omega\in A$. In combinatorics this is often called a monotone graph property and a very active research field in the last years has been the study of the simplicial complex of graphs that do not realize a certain monotone event. See e.g \cite{BBLSW} or \cite{JJ}. 

A fundamental tool is the classical Harris inequality from 1960. 
%For a wealth of interesting generalizations of Harris inequality see \cite{??}. 

\begin{theorem}\label{T:Harris}\cite{H} For edge percolation, two increasing events $A,B$ are always positively correlated, i.e. 
\[P(A)P(B)\le P(A,B).
\]
\end{theorem}

For random orientations, the model $O$, this concept is not directly applicable (but often is for the model $D^p$).
 We will therefore define a related concept for the model $O$. The space of realizations for $O$ is better described as $\Omega_O=\{-1,1\}^{E}$, for some choice of $-1,1$ corresponding to the two different directions for every edge. Define an event $A$ to be {\bf $s$-out-cluster increasing} if for two realizations $\omega, \omega'\in \Omega_O$, $\hpil{C}_s(\omega)\supseteq \hpil{C}_s(\omega')$ and $\omega'\in A$ implies $\omega\in A$. Of course, one may define $s$-in-cluster increasing in the same way and all theorems below would be equally valid for $\vpil{C}_s$. A typical example of an $s$-out-cluster increasing event is $\{s\to a\}$, i.e. $a\in \hpil{C}_s$.
The corresponding definition for $E^p$ was considered in \cite{vdBHK} with the important difference that they thought of the cluster as a set of edges, whereas in this paper it is a set of vertices.

We may now formulate an inequality corresponding to Harris inequality.

\begin{theorem}\label{T:OrientedHarris} Given a vertex $s$ of the finite graph $G$. In the model $O$, two $s$ out-cluster increasing events $A,B$ are always positively correlated, i.e. 
\[ P_O(A)P_O(B)\le P_O(A, B).
\]
\end{theorem}

Theorem \ref{T:OrientedHarris} will follow from the more general Theorem \ref{T:OrientedvdBHK} below, which is very closely inspired by and follows from Theorem 1.1 for $E^p$ in \cite{vdBHK}.

\begin{theorem}\label{T:OrientedvdBHK} Let $G=(V,E)$ be a finite graph and $s\in V(G)$. Let also $A,B$ be $s$-out-cluster increasing events in the model $O$. Then for any $X,Y\subseteq V\setminus \{s\}$,
\[P_O(A,\hpil{C}_s\cap X=\emptyset)P_O(B,\hpil{C}_s\cap Y=\emptyset)\le P_O(A,B,\hpil{C}_s\cap X\cap Y=\emptyset)P_O(\hpil{C}_s\cap (X\cup Y)=\emptyset).
\]
\end{theorem}

\begin{proof}
By Lemma \ref{L:equal1} we get that $P_{E^{1/2}}(A,{C}_s\cap X=\emptyset)=P_{O}(A,\hpil{C}_s \cap X=\emptyset)$ and similarly for the other probabilities in the theorem. Theorem 1.1 in \cite{vdBHK} is the same statement but in $E^p$ with the important difference that they considered a cluster to be the edges that belonged to a path from $s$. However, if an event is $s$-cluster increasing in our sense then it is clearly also $s$-cluster increasing when the cluster is thought of a  set of edges since it is a strictly stronger property. Hence, 
\[P_{E^{p}}(A,{C}_s\cap X=\emptyset)P_{E^{p}}(B,{C}_s\cap Y=\emptyset)\le P_{E^{p}}(A,B,{C}_s\cap X\cap Y=\emptyset)P_{E^{p}}({C}_s\cap (X\cup Y)=\emptyset)
\]
and the theorem follows.
\end{proof}
We get Theorem \ref{T:OrientedHarris} from Theorem \ref{T:OrientedvdBHK} by setting $X=Y=\emptyset$.

\medskip
Let us now consider some examples of usage of these two theorems. Let $a,b,s,t$ be vertices of $G$.
We are interested in the event $\{s \to a\}$, which is the same thing as $\{a\in \hpil{C}_s(G)\}$.

\begin{corollary}\label{C:Twopaths} Consider the model for random orientation $O$. In any finite graph $G$ the events $\{a\in \hpil{C}_s(G)\}$ and $\{b\in \hpil{C}_s(G)\}$ are positively correlated, that is
\[P_O(a\in \hpil{C}_s(G))P_O(b\in \hpil{C}_s(G))\le P_O(a,b\in \hpil{C}_s(G)).
\]
\end{corollary}

\begin{proof}
Set $A=\{a\in \hpil{C}_s(G)\}$ and $B=\{b\in \hpil{C}_s(G)\}$ in the oriented version of Harris inequality, Theorem \ref{T:OrientedHarris}.
\end{proof}

Corollary \ref{C:Twopaths} is true also for the model $D^p$, but we do not need to refer to Theorem \ref{T:OrientedHarris} since it follows directly by Harris inequality.

In the beautiful paper \cite {vdBK}, Jacob van den Berg and Jeff Kahn prove that for any graph under the model $E^p$ the events $a\in {C}_s(G)$ and $b\in {C}_s(G)$ are positively correlated also if one first conditions on $t\notin C_s(G)$. This is far from intuitively clear, in fact the authors claim that it is non-intuitive \cite{vdBK}. The oriented version is the following.

\begin{corollary}\label{C:vdBK} Consider the model for random orientation $O$ and let $G$ be any graph and $a,b,s,t$ vertices of $G$, $s\neq t$. Conditioned on the event that $\{t\notin \hpil{C}_s\}$, the events $\{a\in \hpil{C}_s\}$ and $\{b\in \hpil{C}_s\}$ are positively correlated. In formulas:
\[P_{O}(a,b\in \hpil{C}_s | t\notin \hpil{C}_s )\ge P_{O}(a\in \hpil{C}_s | t\notin \hpil{C}_s )P_{O}(b\in \hpil{C}_s | t\notin \hpil{C}_s ).
\]
\end{corollary}

\begin{proof}
Set $A=\{a\in \hpil{C}_s(G)\}$ and $B=\{b\in \hpil{C}_s(G)\}$ and $X=Y=\{t\}$ in Theorem \ref{T:OrientedvdBHK}.
\end{proof}

The paper \cite{vdBHK} contains several other interesting theorems on correlation and association, most of which can be given versions for random orientations. We end with an example to illustrate the usage of Lemma \ref{L:equal2}.

\begin{theorem}\label{T:vdBHK2} Let $G$ be any graph and $a,b,s,t$ vertices of $G$, $s\neq t$. Conditioned on the event that $\{t\notin \hpil{C}_s\}$ the events $\{a\in \vpil{C}_t\}$ and $\{b\in \hpil{C}_s\}$ are negatively correlated. In formulas:
\[P_{O}(a\in \vpil{C}_t, b\in \hpil{C}_s | t\notin \hpil{C}_s )\le P_{O}(a\in \vpil{C}_t | t\notin \hpil{C}_s )P_{O}(b\in \hpil{C}_s | t\notin \hpil{C}_s ).
\]
\end{theorem}

\begin{proof}
By Lemma \ref{L:equal2} we know that the joint distribution $P_{O}(\vpil{C}_t, \hpil{C}_s)$ is equal to 
$P_{E^p}({C}_t, {C}_s)$ as long as we assume the clusters to be disjoint, which is what is conditioned in the theorem. In \cite{vdBHK} formula (2) is 
$P_{E^p}(a\in {C}_t, b\in {C}_s | t\notin {C}_s )\le P_{E^p}(a\in {C}_t | t\notin {C}_s )P_{E^p}(b\in {C}_s | t\notin {C}_s )$ and the theorem follows.
\end{proof}

It is interesting to notice that if we do not condition on $s\nto t$, then $\{a\in \vpil{C}_t\}$ and $\{b\in \hpil{C}_s\}$ are positively correlated in the model $D^p$ by the ordinary Harris inequality, whereas in the model $O$ it will depend on the graph $G$ if they are positively or negatively correlated.

\end{document}